\documentclass[11pt, a4paper]{article}

\usepackage{amsmath}
\usepackage{amssymb}
\usepackage{amsfonts}
\usepackage{amsthm}
\usepackage{thmtools}
\usepackage{graphicx}
\usepackage{fullpage}
\usepackage[colorlinks=true,linkcolor=blue,citecolor=blue]{hyperref}
\usepackage{cleveref}
\usepackage{cite}
\usepackage{color}
\usepackage{xspace}
\usepackage{enumitem}

\declaretheorem[style=plain,numberwithin=section,name=Theorem]{theorem}
\declaretheorem[style=plain,sibling=theorem,name=Lemma]{lemma}
\declaretheorem[style=plain,sibling=theorem,name=Proposition]{proposition}
\declaretheorem[style=plain,sibling=theorem,name=Corollary]{corollary}
\declaretheorem[style=plain,sibling=theorem,name=Claim]{claim}

\declaretheorem[style=definition,sibling=theorem,name=Definition,qed=$\blacksquare$]{definition}
\declaretheorem[style=definition,sibling=theorem,name=Example,qed=$\blacksquare$]{example}

\declaretheorem[style=definition,sibling=theorem,name=Remark,qed=$\blacksquare$]{remark}

\newcommand{\Z}{\mathbb{Z}}
\newcommand{\K}{\mathcal{K}}
\newcommand{\F}{\mathcal{F}}
\newcommand{\MKtMP}{{\sc Maximum ${\cal K}$-Free $t$-Matching Problem}\xspace}
\newcommand{\KbFP}{{\sc ${\cal K}$-Free $b$-Factor Problem}\xspace}
\newcommand{\MKbMP}{{\sc Maximum ${\cal K}$-Free $b$-Matching Problem}\xspace}
\newcommand{\BEC}{{\sc Boolean Edge-CSP}}
\newcommand{\BECG}{{\sc Boolean Edge-CSP}$(\Gamma_{{\rm cp}\text{-}{\rm jump}})$\xspace}
\newcommand{\BC}{{\sc Boolean CSP}}
\newcommand{\cond}{{\rm RD}}

\title{Finding a Maximum Restricted $t$-Matching\\via Boolean Edge-CSP}

\author{Yuni Iwamasa\thanks{Kyoto University.
E-mail: iwamasa@i.kyoto-u.ac.jp}
\and
Yusuke Kobayashi\thanks{Kyoto University.
E-mail: yusuke@kurims.kyoto-u.ac.jp}
\and
Kenjiro Takazawa\thanks{Hosei University.
E-mail: takazawa@hosei.ac.jp}
}

\date{}

\begin{document}
\maketitle
\begin{abstract}
The problem of finding a maximum $2$-matching without short cycles has received significant attention due to its relevance to the Hamilton cycle problem. 
This problem is generalized to finding a maximum $t$-matching which excludes  
specified complete $t$-partite subgraphs, where $t$ is a fixed positive integer. 
The polynomial solvability of this generalized problem remains an open question.  
In this paper, we present polynomial-time algorithms for the following two cases of this problem: 
in the first case the forbidden complete $t$-partite subgraphs are edge-disjoint; 
and in the second case the maximum degree of the input graph is at most $2t-1$. 
Our result for the first case extends the previous work of Nam (1994)
showing the polynomial solvability of the 
problem of finding a maximum $2$-matching without cycles of length four, 
where the cycles of length four are vertex-disjoint. 
The second result expands upon the works of B\'{e}rczi and V\'{e}gh (2010) and Kobayashi and Yin (2012), 
which focused on graphs with maximum degree at most $t+1$. 
Our algorithms are obtained from exploiting the discrete structure of restricted $t$-matchings and employing an algorithm for the Boolean edge-CSP.
\begin{description}
    \item[Keywords] Polynomial algorithm, $C_k$-free $2$-matching, Jump system, Boolean edge-CSP
\end{description}
\end{abstract}

\section{Introduction}
\label{sec:intro}

The matching problem and its generalizations have been one of the most primary topics 
in combinatorial optimization, and have been the subject of a large number of studies. 
A typical generalization of a matching is 
a {\em $t$-matching} for an arbitrary positive integer $t$: 
an edge subset $M$ in a graph is a $t$-matching\footnote{Such an edge set is 
sometimes called a {\em simple $t$-matching} in the literature, 
but we omit the adjective ``simple'' because 
in this article a $t$-matching is always an edge subset and we never put multiplicities on the edges.} 
if each vertex is incident to at most $t$ edges in $M$. 

While the problem of finding a $t$-matching of maximum cardinality can be solved in polynomial time by a matching algorithm, 
the problem becomes much more difficult, 
typically NP-hard, when additional constraints are imposed. 
The constraints discussed in 
this paper is to exclude certain subgraphs. 
Let $G=(V, E)$ be a graph and 
let ${\cal K}$ be a family of subgraphs of $G$. 
For a subgraph $K$ of $G$, 
let $V(K)$ and $E(K)$ denote 
the vertex set and the edge set of $K$, 
respectively. 
\begin{definition}
An edge subset $M \subseteq E$ is \emph{${\cal K}$-free} if 
$E(K) \not\subseteq M$ for any $K \in {\cal K}$. 
\end{definition}
The problem formulated below is the central issue in this paper, 
whose 
relevance 
will be described in detail in Section \ref{SECtmatching}.

\paragraph{\underline{\MKtMP}}
Given a graph $G=(V, E)$ and a family $\K$ of subgraphs of $G$, 
find a $\K$-free $t$-matching $M\subseteq E$ of maximum cardinality.  

\medskip

Our 
primary contribution are the following two theorems,  showing the polynomial solvability of 
certain classes of 
\MKtMP.
The first result concerns the case where 
${\cal K}$ is an edge-disjoint family of \emph{$t$-regular complete partite subgraphs} of $G$. 
While we defer the definition to Section \ref{sec:notation}, 
here we remark that 
a complete graph 
$K_{t+1}$ and 
a complete bipartite graph 
$K_{t,t}$ are examples of 
a $t$-regular complete partite graph.

\begin{theorem}
\label{thm:disjointtmatching}
For a fixed positive integer $t$, 
\MKtMP
can be solved in polynomial time 
if 
all the subgraphs in ${\cal K}$ 
are $t$-regular complete partite and pairwise edge-disjoint. 
\end{theorem}

In the second result, 
instead of the edge-disjointness of the subgraphs in $\mathcal{K}$, 
we assume that 
the maximum degree of the input graph $G$ is bounded. 

\begin{theorem}
\label{thm:degreetmatching}
For a fixed positive integer $t$, 
\MKtMP
can be solved in polynomial time 
if 
all the subgraphs in 
${\cal K}$ are $t$-regular complete partite 
and 
the maximum degree of $G$ is at most $2t-1$. 
\end{theorem}

Theorems~\ref{thm:disjointtmatching} and~\ref{thm:degreetmatching} 
offer larger polynomially solvable classes of 
\MKtMP
than the previous work introduced in \Cref{SECtmatching} below. 
In addition, 
we will describe 
the relevance of Theorems \ref{thm:disjointtmatching} and \ref{thm:degreetmatching} to the literature, 
together with 
their extensions and variants in the subsequent sections. 
Here we just remark that 
the assumption on the complete partiteness of 
the forbidden subgraphs 
in Theorems~\ref{thm:disjointtmatching} and~\ref{thm:degreetmatching} 
is unavoidable, 
because 
the problem is NP-hard without this assumption (see Proposition \ref{PROPNP-hard} below).

\subsection{Previous Work on Restricted $t$-Matchings}
\label{SECtmatching}

\MKtMP
has its origin in the case
 where $t=2$ and ${\cal K}$ is composed of short cycles.  
Let 
$k$ be a positive integer. 
If 
${\cal K}$ is the set of all cycles of length at most $k$, 
then 
a $\mathcal{K}$-free $2$-matching 
is referred to as a \emph{$C_{\le k}$-free $2$-matching}, 
and 
{\sc Maximum ${\cal K}$-Free $2$-Matching Problem}
as 
the \emph{$C_{\le k}$-free $2$-matching problem}. 
Similarly, 
if ${\cal K}$ is the set of all cycles of length exactly $k$, 
then 
a $\mathcal{K}$-free $2$-matching 
is referred to as a 
\emph{$C_{k}$-free $2$-matching}, 
and 
{\sc Maximum ${\cal K}$-Free $2$-Matching Problem}
as 
the \emph{$C_{k}$-free $2$-matching problem}. 
The 
$C_{\le k}$-free and $C_k$-free $2$-matching problems 
have attracted significant attention 
because of 
their relevance to the Hamilton cycle problem;
for $k \ge |V|/2$, a $C_{\le k}$-free $2$-matching of cardinality $|V|$ is a Hamilton cycle.
When $k$ is small, the $C_{\le k}$-free $2$-matching problem is not directly used to find Hamilton cycles, 
but 
it can be applied to designing 
approximation algorithms for related problems such as the graph-TSP and the minimum $2$-edge-connected spanning subgraph problem. For example, in a recent paper~\cite{KN2023}, an approximation algorithm for the minimum 2-edge-connected spanning subgraph problem is provided using a maximum $C_{\le 3}$-free $2$-matching.

The complexity of 
the $C_{\le k}$-free $2$-matching problem depends on the value of $k$. 
It is straightforward to see that 
this problem 
can be solved in polynomial time 
for $k\le 2$. 
For $k=3$, 
Hartvigsen~\cite{HartD} gave
a polynomial-time algorithm for the $C_{\le 3}$-free $2$-matching problem. 
For $k \ge 5$, 
Papadimitriou proved the NP-hardness of 
the $C_{\le k}$-free $2$-matching problem  (see \cite{CP80}). 

For the case $k=4$, 
it is open whether 
the $C_{\le 4}$-free and $C_4$-free $2$-matching problems can be solved in polynomial time, 
and these problems have rich literature of polynomial-time algorithms for several special cases. 
First, 
for subcubic graphs, i.e., graphs with maximum degree at most three, 
polynomial-time algorithms for the $C_{4}$-free and the $C_{\le 4}$-free $2$-matching problems were 
given by B\'{e}rczi and Kobayashi~\cite{BK12} and B\'{e}rczi and V\'{e}gh~\cite{BV10}, respectively. 
Simpler algorithms for both problems in subcubic graphs (and for some of their weighted variants) were designed by Hartvigsen and Li~\cite{HL11} and 
by Paluch and Wasylkiewicz~\cite{PW21IPL}.
It is worth noting that a connection between the $C_{4}$-free matching problem and a connectivity augmentation problem is highlighted in~\cite{BK12},
underscoring the significance of the $C_{4}$-free matching problem.
Second, 
for the graphs in which the cycles of length four are vertex-disjoint, 
Nam~\cite{Nam94} gave a polynomial-time algorithm for the $C_{4}$-free $2$-matching problem. 
Finally, 
for bipartite graphs, several of polynomial-time algorithms 
are devised; see Section~\ref{sec:related} for details. 

Let $t$ be an arbitrary positive integer. 
The 
$C_{k}$-free $2$-matching problem is   
generalized to 
\MKtMP
for general $t$ 
in the following way. 
Let $K_t$ denote the complete graph with $t$ vertices, 
and 
$K_{t,t}$ the complete bipartite graph in which each color class has $t$ vertices. 
Here, note that a cycle of length three is isomorphic to $K_3$. 
Thus, 
the $C_{3}$-free $2$-matching problem can be naturally generalized to 
\MKtMP,
where ${\cal K}$ is the set of all subgraphs that are isormorphic to $K_{t+1}$. 
We refer to this special case of 
\MKtMP
as the \emph{$K_{t+1}$-free $t$-matching problem}.  
Similarly, 
by noting that 
a cycle of length four is isomorphic to $K_{2,2}$, 
we can generalize 
the $C_{4}$-free $2$-matching problem 
to the 
\emph{$K_{t,t}$-free $t$-matching problem}. 
This  
is another special class of 
\MKtMP,
where ${\cal K}$ is the set of all subgraphs isormorphic to $K_{t,t}$.

The polynomial solvability of these two problems are open. 
For 
certain special cases of 
\MKtMP,
however, 
several polynomial-time algorithms 
are presented, corresponding to those for 
the $C_{\le k}$-free and $C_k$-free $2$-matching problems. 
First, 
B\'{e}rczi and V\'{e}gh~\cite{BV10} gave a polynomial-time algorithm for 
\MKtMP
for the case where ${\cal K}$ consists of $K_{t+1}$'s and $K_{t,t}$'s 
and the input graph $G$ has maximum degree at most $t+1$. 
This extends that for 
the $C_{\le 4}$-free $2$-matching problem 
in subcubic graphs.
Second, 
Kobayashi and Yin~\cite{KY12} presented a polynomial-time algorithm for 
\MKtMP 
for 
the case where 
${\cal K}$ consists of all the subgraphs isomorphic to a fixed $t$-regular complete partite graph 
and the input graph $G$ has maximum degree at most $t+1$. 
Kobayashi and Yin \cite{KY12} also proved that this assumption on $\mathcal{K}$ 
is inevitable. 
\begin{proposition}[follows from Kobayashi and Yin \cite{KY12}]
\label{PROPNP-hard}
If $H$ is a connected $t$-regular graph which is not complete partite 
and ${\cal K}$ is the set of all subgraphs isomorphic to $H$, then 
\MKtMP
is NP-hard even 
when  
the maximum degree of $G$ is at most $t+1$
and 
the subgraphs in $\mathcal{K}$ are 
pairwise edge-disjoint.
\end{proposition}
As mentioned above, 
this NP-hardness explains that 
the assumption on the complete partiteness of the forbidden subgraphs is also unavoidable in
Theorems \ref{thm:disjointtmatching} and \ref{thm:degreetmatching}.

Finally, 
for the $K_{t,t}$-free $t$-matching problem in bipartite graphs, 
some polynomial-time algorithms are designed,  extending those for the $C_4$-free $2$-matching problem in bipartite graphs (see Section~\ref{sec:related}).

\subsection{Our Contribution}
\label{sec:ourresults}

We have seen that the polynomial solvability of the $K_{t+1}$-free $t$-matching and $K_{t,t}$-free $t$-matching problems is unkown. 
As well as these problems, 
the polynomial solvability of 
\MKtMP 
in general graphs 
for ${\cal K}$ being 
an arbitrary 
family of $t$-regular complete partite subgraphs is unknown. 
The contribution of this paper is to present polynomial-time algorithms for several special cases of this problem.  

\subsubsection{Overview of Our Results}

Recall our first result, Theorem \ref{thm:disjointtmatching}, 
solving the case where 
${\cal K}$ is an edge-disjoint family of $t$-regular complete partite subgraphs of $G$. 
By setting $t=2$ in Theorem \ref{thm:disjointtmatching}, we immediately obtain the following corollary. 

\begin{corollary}
\label{cor:edgedisjointC3C4}
{\sc Maximum ${\cal K}$-Free $2$-Matching Problem} 
can be solved in polynomial time 
if all the subgraphs in ${\cal K}$ are 
isomorphic to $C_3$ or $C_4$, 
and 
are 
pairwise edge-disjoint. 
\end{corollary}

Corollary \ref{cor:edgedisjointC3C4} extends the result by Nam~\cite{Nam94}, 
solving the $C_{4}$-free $2$-matching problem where the cycles of length four are vertex-disjoint. 
Namely, Corollary~\ref{cor:edgedisjointC3C4} 
extends 
vertex-disjointness 
to edge-disjointness, 
and allows ${\cal K}$ to include not only $C_4$ but also $C_3$.

Next, 
recall our second result, Theorem \ref{thm:degreetmatching}, 
which solves the case where
the maximum degree of the input graph is at most $2t-1$. 
Theorem \ref{thm:degreetmatching} expands upon the works of B\'{e}rczi and V\'{e}gh~\cite{BV10} and Kobayashi and Yin~\cite{KY12},  
which focused graphs with maximum degree at most $t+1$. 
That is, Theorem~\ref{thm:degreetmatching} improves the degree bound from $t+1$ to $2t-1$, 
where $2t-1 > t+1$ if $t > 2$.  

We further present some extensions of Theorems \ref{thm:disjointtmatching} and \ref{thm:degreetmatching}. 
Below is one extension of Theorem \ref{thm:disjointtmatching}, 
which will be used in our proof for Theorem \ref{thm:degreetmatching}. 
The pairwise edge-disjointness of the subgraphs in $\mathcal{K}$ 
is relaxed to the following condition: 
\begin{enumerate}[label=(\cond),ref=\cond]
\item \label{cond:star} 
The subgraph family $\mathcal{K}$ is partitioned into subfamilies $\mathcal{K}_1, \dots , \mathcal{K}_\ell$ such that 
    \begin{itemize}
        \item for each subfamily $\mathcal{K}_i$ ($i=1,\ldots, \ell$), the number $\left|\bigcup_{K \in \mathcal{K}_i} V(K)\right|$ of its vertices is bounded by a fixed constant, and  
        \item for distinct subfamilies $\mathcal{K}_i$ and $\mathcal{K}_j$ ($i,j \in \{1, \dots , \ell\})$ and for 
        each pair of subgraphs $K \in \mathcal{K}_i$ and $K' \in \mathcal{K}_j$, 
        it holds that $K$ and $K'$ are edge-disjoint. 
    \end{itemize}
\end{enumerate}
Here ``\cond'' stands for ``Relaxed Disjointness.''

\begin{theorem}
\label{thm:startmatching}
For a fixed positive integer $t$,  
\MKtMP 
can be solved in polynomial time 
if $\K$ is a family of $t$-regular complete partite subgraphs of $G$ satisfying the condition~\eqref{cond:star}. 
\end{theorem}

Other results include extensions from $t$-matchings to $b$-matchings (Theorems \ref{thm:disjointbfactor},  \ref{thm:disjointbmatching}, 
\ref{thm:starbmatching}, 
\ref{thm:degreebfactor}, 
and 
\ref{thm:degreebmatching}). 
For a vector $b \in \Z^V$, 
a \emph{$b$-matching} is an edge subset $M\subseteq E$ 
such that 
each vertex $v\in V$ is incident to at most $b(v)$ edges in $M$.
Namely, we can deal with inhomogeneous degree constraints.
Moreover, we provide 
an extension from forbidding subgraphs to
forbidding degree sequences (Theorem \ref{thm:generalization}).
In particular, 
the latter extension offers some new results on restricted $2$-matchings (Examples \ref{EX01} and \ref{EX02}).

\subsubsection{Technical Ingredients}
\label{subsubsec:tech}

Technically, 
our algorithms are established by exploiting two important previous results, 
one is on 
the discrete structure of $\mathcal{K}$-free $t$-matchings
and 
the other 
is on
the constraint satisfaction problem (CSP)\@. 
This is in contrast to the fact that 
the previous algorithms  \cite{BV10,KY12,Nam94} are 
based on graph-theoretical methods.

The first result  
is outlined as follows. 
Let $b \in \Z^V$ with $b(v)\le t$ for each $v\in V$
and 
let $J \subseteq \Z^V$ be the set of the degree sequences of all $\K$-free $b$-matching in $G$. 
Kobayashi, Szab\'o, and Takazawa \cite{KST12} proved that 
$J$ forms a \emph{constant-parity jump system} if
all the subgraphs in $\K$ are $t$-regular complete partite  
(see \Cref{thm:parititejump} below).
Here a 
constant-parity 
jump system is a subset of $\Z^V$, 
which offer a discrete structure 
generalizing matroids; 
see \Cref{sec:jumpsystem} for the definition. 

The second is on the polynomial-time solvability of a class of the CSP\@. 
The \emph{Boolean edge-CSP} is the problem of finding an edge subset $M \subseteq E$ of a given graph $G=(V,E)$
such that 
the set of edges in $M$ incident to each vertex $v\in V$ 
satisfies a certain constraint 
associated with $v$; see \Cref{sec:BooleanEdgeCSP} for 
formal description. 
While the Boolean edge-CSP is NP-hard in general,
Kazda, Kolmogorov, and Rol{\'{\i}}nek~\cite{KKR19} showed that this problem can be solved in polynomial time if the constraint associated with $v$ is described by a constant-parity jump system for each $v\in V$
(see \Cref{thm:EdgeCSP} below).

The most distinctive part of this paper is a reduction of
\MKtMP 
to the Boolean edge-CSP\@. 
It appears in the proof of Theorem \ref{thm:disjointbfactor} below, 
which deals with the problem of finding a $\K$-free $b$-factor, 
i.e.,\ a $t$-matching with specified degree sequence $b \in \Z^V$. 
Here, 
on the basis of the relationship between $\K$-free $b$-matchings and jump systems 
(\Cref{thm:parititejump}), 
we construct 
a polynomial reduction 
of the problem of finding a $\K$-free $b$-factor 
to the Boolean edge-CSP with  constant-parity jump system constraints. 

Theorem \ref{thm:disjointtmatching} is then derived from Theorem \ref{thm:disjointbfactor}. 
In order to prove Theorem \ref{thm:disjointtmatching}, 
we iteratively solve subproblems of finding a $\K$-free $b$-factor.
We remark that constant-parity jump systems play a key role here, 
as well as the reduction mentioned above. 
The fact that $J$ is a constant-parity jump system guarantees that
the number of the iterations is polynomially bounded by the input size (see Lemma \ref{lem:factor2matching} below).

Theorem \ref{thm:startmatching} is proved in the same manner. 
We then derive Theorem \ref{thm:degreetmatching} from 
Theorem \ref{thm:startmatching} 
by constructing 
a subfamily $\K' \subseteq \K$ such that
$\K'$ satisfies \eqref{cond:star}, 
a $\K'$-free $t$-matching exists in $G$ 
if and only if 
a $\K$-free $t$-matching exists in $G$, and
we can construct a $\K$-free $t$-matching from a $\K'$-free $t$-matching
in polynomial time.

\subsection{Further Related Work}
\label{sec:related}

The $C_4$-free $2$-matching problem has been actively studied 
in the setting when the input graph is restricted to be bipartite. 
Hartvigsen~\cite{Hart06}, Kir\'{a}ly~\cite{Kir99,Kir09}, and Frank~\cite{Fra03} gave
min-max theorems for the $C_4$-free $2$-matching problem in bipartite graphs, and 
more generally for the $K_{t,t}$-free $t$-matching problem in bipartite graphs, 
which implies the polynomial solvability of the problems. 
To the best of our knowledge, 
Frank \cite{Fra03} 
is the first work 
to generalize the restricted $2$-matching in this context to $t$-matchings. 
For the $C_4$-free $2$-matching problem in bipartite graphs, 
Hartvigsen~\cite{Hart06} and Pap~\cite{Pap07} designed combinatorial polynomial-time algorithms, 
Babenko~\cite{Bab12} improved the running time, and 
Takazawa~\cite{Tak17DAM} showed a decomposition theorem. 
Takazawa~\cite{Tak17DO,Tak22} extended these results to more generalized classes of 
\MKtMP.

The weighted variant of 
\MKtMP 
has also attracted much attention. 
In the weighted problem, an input consists of a graph, a family ${\cal K}$ of subgraphs, and a non-negative weight function on the edge set, and 
the objective is to find a ${\cal K}$-free $t$-matching with maximum total weight. 
It is shown by Kir\'{a}ly (see \cite{Fra03}) and by B\'{e}rczi and Kobayashi \cite{BK12} that 
the weighted $C_{\le 4}$-free $2$-matching problem
is NP-hard even if the input graph is restricted to be cubic, bipartite, and planar. 
For the weighted $C_{4}$-free $2$-matching problem in bipartite graphs, and 
more generally for the weighted $K_{t,t}$-free $t$-matching problem in bipartite graphs,  
under the assumption that the weight function satisfies a certain property, 
Makai~\cite{Mak07} gave a polyhedral description,  
Takazawa~\cite{Tak09} designed a combinatorial polynomial-time algorithm, 
and Paluch and Wasylkiewicz~\cite{PW21ESA} presented a faster and simpler algorithm. 

It is still open whether the weighted $C_{3}$-free $2$-matching problem can be solved in polynomial time. 
For the weighted $C_{3}$-free $2$-matching problem in subcubic graphs, 
Hartvigsen and Li~\cite{HL13} gave a polyhedral description and a polynomial-time algorithm, 
and faster polynomial-time algorithms were presented by Kobayashi~\cite{Kob10} and by Paluch and Wasylkiewicz~\cite{PW21IPL}. 
Recently, 
Kobayashi~\cite{Kob22} designed a polynomial-time algorithm for the 
weighted $C_{3}$-free $2$-matching problem in which the cycles of length three are edge-disjoint. 

The relationship between ${\cal K}$-free $t$-matchings and jump systems has been studied in~\cite{BK12,Cun02,KST12}, 
some of which will be used in this paper.  
More generally, the relationship between weighted ${\cal K}$-free $t$-matchings and discrete convexity 
has been studied in~\cite{BK12,Kob10,Kob14,KST12}.

\subsection{Organization}

The rest of the paper is organized as follows. 
In Section \ref{sec:pre}, 
we present the basic definitions and results in a formal manner. 
In Section \ref{sec:edgedisjoint}, 
we solve the problem under the assumption that the subgraphs in $\K$ are pairwise edge-disjoint, 
and then 
under the relaxed condition (\ref{cond:star}). 
Section \ref{sec:boundeddegree} is devoted to a solution to the graphs with maximum degree at most $2t-1$. 
Finally, 
in Section \ref{sec:generalization}, 
we deal with a more generalized problem where the forbidden structure is described in terms of degree sequences.

\section{Preliminaries}
\label{sec:pre}

Let $\mathbb{Z}_+$ denote the set of nonnegative integers, 
and ${\bf 0}$ (resp.~${\bf 1}$) denote the all-zero (resp.~all-one) vector of appropriate dimension. 
For a finite set $V$, its subset $U \subseteq V$, and 
a vector $x \in \Z^V$, let $x(U) = \sum_{v \in U} x(u)$.

\subsection{Basic Definitions on Graphs}
\label{sec:notation}

Throughout this paper, we assume that graphs have no self-loops to simplify the description, while they may have parallel edges. 
Let $G=(V, E)$ be a graph. 
For a subgraph $H$ of $G$, let $V(H)$ and $E(H)$ denote the vertex set and edge set of $H$, respectively. 
For a vertex set $X \subseteq V$, let $G[X]$ denote the subgraph induced by $X$.

Let $F \subseteq E$ be an edge subset and let $v \in V$ be a vertex.  
The set of edges in $F$ incident to $v$ is denoted by $\delta_F(v)$. 
If $F = E(H)$ for some subgraph $H$ of $G$, 
then 
$\delta_{E(H)}(v)$ is often abbreviated as $\delta_H(v)$. 
When no confusion arises, 
$\delta_G(v)$ is further abbreviated as $\delta(v)$. 
The number of edges incident to $v$, 
i.e.,\ $|\delta(v)|$, 
is referred to as the \emph{degree} of $v$. 
The \emph{degree sequence} $d_F$ of $F \subseteq E$ is a vector in $\mathbb{Z}_+^V$ defined by 
$d_F(u) = |\delta_F(u)|$ 
for each $u\in V$.

For a positive integer $t$, a graph is called \emph{$t$-regular} if every vertex has degree $t$. 
A graph $G=(V, E)$ is said to be a \emph{complete partite graph} 
if there exists a partition $\{V_1, \dots , V_p\}$ of $V$
such that $E = \{ uv \colon u \in V_i, v \in V_j, i\neq j\}$ for some positive integer $p$. 
In other words, a complete partite graph is the complement of the disjoint union of complete graphs.
Each $V_i$ is called a \emph{color class} of $G$.

As defined in Section \ref{sec:intro}, 
for a positive integer $t$, an edge set $M \subseteq E$ is called a \emph{$t$-matching}  
if $d_M(v) \le t$ for every $v \in V$. 
In particular, 
if $d_M(v) = t$ holds for every $v \in V$, 
then $M$ is called a \emph{$t$-factor}. 
For a vector $b \in \mathbb{Z}_+^V$, an edge set $M \subseteq E$ is called a \emph{$b$-matching} (resp.~\emph{$b$-factor}) 
if $d_M(v) \le b(v)$ (resp.~$d_M(v) = b(v)$) for every $v \in V$.

In what follows, 
instead of 
\MKtMP,
we deal with the following slightly generalized problems.

\paragraph{\underline{\KbFP}}
Given a graph $G=(V, E)$, $b \in \mathbb{Z}_+^V$, and a family $\K$ of subgraphs of $G$, 
find a $\K$-free $b$-factor (if one exists).

\paragraph{\underline{\MKbMP}}
Given a graph $G=(V, E)$, $b \in \mathbb{Z}_+^V$, and a family ${\cal K}$ of subgraphs of $G$, 
find a $\K$-free $b$-matching with maximum cardinality.

\medskip

Note that 
\MKtMP 
is a special case of 
\MKbMP,
where $b(v) = t$ for each $v \in V$.

\begin{remark}
In this paper, we only consider the case where ${\cal K}$ consists of subgraphs of size bounded by a fixed constant (e.g., $t$-regular complete partite subgraphs for a fixed integer $t$). 
In such a case, since $|{\cal K}|$ is polynomially bounded by the size of the input graph, 
the representation of ${\cal K}$ does not affect the polynomial solvability of the problem.  
Therefore, in what follows, we suppose that ${\cal K}$ is explicitly given as the list of its elements. 
\end{remark}

\begin{remark}\label{rmk:connected}
Let $G = (V, E)$ be a graph, $b \in \Z_+^V$ with $b(v) \leq t$ for each $v \in V$, and $K$ a connected $t$-regular subgraph of $G$.
We can easily observe that, if a $b$-matching $M \subseteq E$ of $G$ contains $K$,
then $K$ forms a connected component of the induced subgraph $(V, M)$ of $G$ by $M$.
\end{remark}

\subsection{Jump System}
\label{sec:jumpsystem}

Let $V$ be a finite set. 
For a subset $U \subseteq V$, 
let $\chi_U \in \{0,1 \}^V$ denote the characteristic vector of $U$, that is, $\chi_U(v) = 1$ for $v \in U$ and
$\chi_U(v)=0$ for $ v \in V \setminus U$. 
If  $U = \{u\}$ for an element $u\in V$, then 
$\chi_{\{u\}}$ is simply denoted by $\chi_u$. 

For two vectors $x, y \in \mathbb{Z}^V$, 
a vector $s \in \mathbb{Z}^V$ is called an \emph{$(x, y)$-increment} if 
$s = \chi_u$ 
and 
$x(u) < y(u)$ 
for some $u \in V$, or 
$s = -\chi_u$ 
and 
$x(u) > y(u)$ 
for some $u \in V$.
A nonempty set $J \subseteq \Z^V$ is said to be a \emph{jump system} if it satisfies the following exchange axiom (see~\cite{BC95}): 
\begin{quote}
    For any $x, y \in J$ and for any $(x, y)$-increment $s$ with $x+s \not\in J$, 
    there exists an $(x+s, y)$-increment $t$ such that $x+ s +t \in J$.
\end{quote}
In particular, 
a jump system $J \subseteq \mathbb{Z}^V$ is called a \emph{constant-parity jump system}
if $x(V) - y(V)$ is even for any $x, y \in J$. 

Constant-parity jump systems include several discrete structures as special classes. 
First, 
for a matroid with a basis family ${\cal B}$, 
it follows from 
the exchange property of matroid bases 
that $\{\chi_B \colon B \in {\cal B}\}$ is a constant-parity jump system. 
Second, the characteristic vectors of all the feasible sets of an even delta-matroid form a constant-parity jump system (see~\cite{BC95}). 
Finally, 
for a graph $G=(V, E)$, 
the set 
$\{d_F \colon F \subseteq E\}$ 
of the degree sequences of all the edge subsets is also a constant-parity jump system. 
See~\cite{BC95,Lov97,Mur06} for details on jump systems.

The following theorem shows a relationship between ${\cal K}$-free $b$-matchings and jump systems.

\begin{theorem}[{follows from \cite[Proposition 3.1]{KST12}}]
\label{thm:parititejump}
Let $G=(V, E)$ be a graph, let $t$ be a positive integer, and 
let $b \in \mathbb{Z}_+^V$ be a vector such that $b(v) \le t$ for each $v \in V$. 
For a family ${\cal K}$ of complete partite $t$-regular subgraphs in $G$, 
the degree sequences of all ${\cal K}$-free $b$-matchings in $G$ form a constant-parity jump system.  
\end{theorem}

\begin{remark} 
Theorem \ref{thm:parititejump} is a modest extension of 
the original statement \cite[Proposition 3.1]{KST12}, in which 
 $b(v) = t$ for each $v \in V$ and  
 ${\cal K}$ is the set of all subgraphs in $G$ that are isomorphic to a graph in a given list of complete partite $t$-regular subgraphs. 
The same proof, however, works for \Cref{thm:parititejump} as well. 
\end{remark}

We here describe a few basic operations on jump systems, which will be used in the subsequent sections. 

\paragraph{Intersection with a box.}
A \emph{box} is a set of the form $\{x \in \mathbb{R}^V \colon \underline{b} \le x \le \overline{b}\}$ 
for some vectors $\underline{b} \in (\mathbb{R} \cup \{- \infty\})^V$ and $\overline{b} \in (\mathbb{R} \cup \{+ \infty\})^V$. 
If $J\subseteq \Z^V$ is a constant-parity jump system, 
then 
the intersection 
$$J \cap \{x \in \mathbb{R}^V \colon \underline{b} \le x \le \overline{b}\}$$
of $J$ and a box is 
also a constant-parity jump system unless it is empty. 

\paragraph{Minkowski sum.}
For two sets $J_1, J_2 \subseteq \mathbb{Z}^V$, 
their \emph{Minkowski sum} $J_1 + J_2$ is a subset of $\mathbb{Z}^V$ defined by 
$$J_1 + J_2 = \{ x + y \colon x \in J_1,\ y \in J_2\}.$$
It was shown by Bouchet and Cunningham~\cite{BC95} that 
the Minkowski sum of two constant-parity jump systems is also a constant-parity jump system. 

\paragraph{Splitting.}
Let $\{ U_v \colon v \in V\}$ be 
a family of nonempty disjoint finite sets indexed by $v \in V$, 
and 
let $U = \bigcup_{v \in V} U_v$. 
For a set $J \subseteq \mathbb{Z}^V$, we define the \emph{splitting} of $J$ to $U$ as 
$$J' = \{x' \in \mathbb{Z}^U \colon  \mbox{$x'(U_v) = x(v)$  for each $v \in V$ for some $x \in J$}\}.$$
The splitting of a constant-parity jump system is also a constant-parity jump system; see~\cite{Mur21,KMT07}.

If the degree sequences 
of all the $\K$-free $b$-matchings form a constant-parity jump system, 
then 
\MKbMP 
reduces 
to 
\KbFP
which is formally stated as follows.

\begin{lemma}
\label{lem:factor2matching}
Let $G = (V, E)$ be a graph,   
${\cal K}$ be a family of subgraphs of $G$, 
and let $b \in \mathbb{Z}^V_+$. 
If the degree sequences of all the ${\cal K}$-free $b$-matchings in $G$ form a constant-parity jump system, 
then 
a ${\cal K}$-free $b$-matching in $G$ with maximum cardinality can be computed by 
testing the existence of a ${\cal K}$-free $b'$-factor in $G$ for polynomially many vectors $b' \in \mathbb{Z}^V_+$ with $b' \le b$. 
\end{lemma}

\begin{proof}
Denote by $J \subseteq \Z^V$ the constant-parity jump system consisting of the degree sequences of all the ${\cal K}$-free $b$-matchings in $G$. 
Given an initial vector in $J$, 
we can maximize a given linear function over $J$ 
by using the membership oracle of $J$
at most polynomially many times  \cite{ST07,BC95,AFN95}. 
Here, the \emph{membership oracle of $J$} is an oracle that answers whether a given vector is in $J$ or not. 

Since an empty edge set is a ${\cal K}$-free $b$-matching, it holds that ${\bf 0} \in J$. 
That is, we can take ${\bf 0}$ as the initial vector in $J$. 
Now the lemma follows because 
accessing the membership oracle of $J$ corresponds to 
testing the existence of a ${\cal K}$-free $b'$-factor in $G$. 
\end{proof}

\subsection{Boolean Edge-CSP}
\label{sec:BooleanEdgeCSP}

The \emph{constraint satisfaction problem} (\emph{CSP}) 
is a fundamental topic in theoretical computer science 
and has been intensively studied in various fields (see, e.g.,~\cite{RBW06}).
In this paper, 
we focus on  the \emph{Boolean edge-CSP}, 
which is 
formulated as follows.

An instance of the Boolean CSP is a pair $(E,\mathcal{C})$, 
where $E$ is the set of Boolean variables and 
$\mathcal{C}$ is that of constraints. 
A constraint $C \in \mathcal{C}$ is a pair $(\sigma_C, R_C)$, 
where the \emph{scope} $\sigma_C \subseteq E$ is the set of the variables appearing in $C$
and 
the \emph{relation} $R_C$ is 
a subset of $\{0,1\}^{\sigma_C}$. 
In general a scope can be a multi-subset of $E$, 
but for notational simplicity we define a scope as a subset of $E$.
The objective 
of the Boolean CSP is to find a mapping $f\colon E \to \{0,1\}$ 
such that
$(f(e))_{e \in \sigma_C} \in R_C$
for each constraint $C\in \mathcal{C}$. 

A central topic of the CSP is the classification of the computational complexity according to the 
relations that can appear in the constraints. 
Let $\Gamma$ denote a set of relations, 
which is referred to as a \emph{language}. 
The problem of finding a solution to a Boolean CSP instance in which every relation appearing in the constraints belongs to $\Gamma$ 
is denoted by \BC$(\Gamma)$.

Schaefer~\cite{Sch78} established a dichotomy theorem stating that 
\BC$(\Gamma)$ is in class P if 
the 
language $\Gamma$ satisfies one of certain six conditions,
and is NP-hard otherwise.
Bulatov~\cite{Bul17} and Zhuk~\cite{Zhu20} independently established 
its generalization, i.e.,
a 
dichotomy theorem for the CSP over any finite domain,
which affirmatively settled a long-standing open question
posed by Feder and Vardi~\cite{FV98}.

By imposing a structural restriction on the set $\mathcal{C}$ of constraints,
we may obtain  
another class of the Boolean CSP which can be solved in polynomial time.
An example is the \emph{Boolean edge-CSP}, 
a class of the Boolean CSP 
in which each variable appears in exactly two constraints. 
An instance $(E,\mathcal{C})$ of the Boolean edge-CSP over the language $\Gamma$, denoted by \BEC($\Gamma$), is 
interpreted in terms of a graph $G=(V,E)$ in the following way. 

The variable set $E$ coincides with the edge set of the graph $G$. 
A mapping $f\colon E\to \{0,1\}$ corresponds to a subset $M\subseteq E$ of edges 
determined by 
$\chi_M(e) = f(e)$ for each $e\in E$. 
One constraint $C=(\sigma_C, R_C) \in \mathcal{C}$ is described by one vertex $v\in V$ 
and the set $\delta(v)$ of its incident edges: 
the scope 
$\sigma_C \subseteq E$ is the edge set $\delta(v)\subseteq E$; 
and 
the relation 
$R_C \subseteq \{0,1\}^{\delta(v)}$ is described by an edge subset family $\mathcal{F}_v \subseteq 2^{\delta(v)}$ 
by $R_C=\{ \chi_F  \colon F \in \F_v \}$. 
Observe that each variable (edge) appears exactly in two constraints (vertices).

\paragraph{\underline{\BEC{\rm $(\Gamma)$}}}
Given a graph $G=(V, E)$ and
  an edge subset family $\F_v \subseteq 2^{\delta(v)}$  
  whose corresponding relation $\{ \chi_F \colon F \in \F_v \}$ belongs to $\Gamma$ 
  for each vertex $v\in V$, 
find an edge set $M \subseteq E$ such that $\delta_M(v) \in {\cal F}_v$ for each $v \in V$ (if one exists). 

\medskip

We remark that the 
relation $\F_v \subseteq 2^{\delta(v)}$   ($v \in V$) is not given by the membership oracles but by the list of the edge subsets, 
and hence the input size is $O(|V| + |E| + \sum_{v \in V} |{\cal F}_v| )$.

For a language $\Gamma$, 
if 
\BC($\Gamma$)
belongs to class P, 
then so is \BEC$(\Gamma)$. 
We thus focus on Boolean languages $\Gamma$ such that \BC$(\Gamma)$ is NP-hard.
Feder~\cite{Fed01} showed that if $\Gamma$ contains the unary relations $\{(0)\}$ and $\{(1)\}$ and a relation that is not a delta-matroid,
then \BEC$(\Gamma)$ is NP-hard. 
On the other hand,
Kazda, Kolmogorov, and Rol{\'{\i}}nek~\cite{KKR19}
proved that 
\BEC($\Gamma$) belongs to class P if 
every relation is an even delta-matroid. 
Since an even delta-matroid can be identified with a constant-parity jump system, with each coordinate in $\{0, 1\}$, 
in what follows in this paper, 
we refer to an even delta-matroid as a constant-parity jump system 
for the unity of terminology.
Let $\Gamma_{{\rm cp}\text{-}{\rm jump}}$ denote the set of all constant-parity jump systems over the Boolean domain. 
\begin{theorem}[Kazda, Kolmogorov, and Rol{\'{\i}}nek~\cite{KKR19}]
\label{thm:EdgeCSP}
\BECG can be solved in polynomial time.
\end{theorem}

\section{Edge-Disjoint Forbidden Subgraphs}
\label{sec:edgedisjoint}

In this section, we consider the case when ${\cal K}$ is an edge-disjoint family of $t$-regular complete partite subgraphs. 
We first give a polynomial-time algorithm for 
\KbFP
by reducing the problem to \BECG in \Cref{thm:disjointbfactor}.
Then, by using this algorithm as a subroutine, we present a polynomial-time algorithm for 
\MKbMP 
(\Cref{thm:disjointbmatching}), 
which implies \Cref{thm:disjointtmatching}. 
Finally, 
we prove the polynomial solvability under the condition (\ref{cond:star}) 
in \Cref{thm:starbmatching}, 
which will be used in the next section. 

\begin{theorem}
\label{thm:disjointbfactor}
For a fixed positive integer $t$, 
\KbFP
can be solved in polynomial time if 
$b(v) \le t$ for each $v \in V$ and 
all the subgraphs in 
${\cal K}$ are $t$-regular complete partite and 
pairwise edge-disjoint. 
\end{theorem}

\begin{proof}
We prove the theorem by constructing a polynomial reduction to \BECG.
Let $(G,b,\K)$ be
an instance of 
\KbFP,
where 
$G=(V,E)$, $b\in \Z_+^V$, and $\cal K$ 
is a family of subgraphs in $G$.

Recall that an input of the Boolean edge-CSP consists of a graph and a constraint on each vertex.
Our input graph $G' = (V', E')$ of the Boolean edge-CSP is constructed as follows (see also Figure~\ref{fig:reduction}):
\begin{figure}
\centering
\includegraphics[width=10cm]{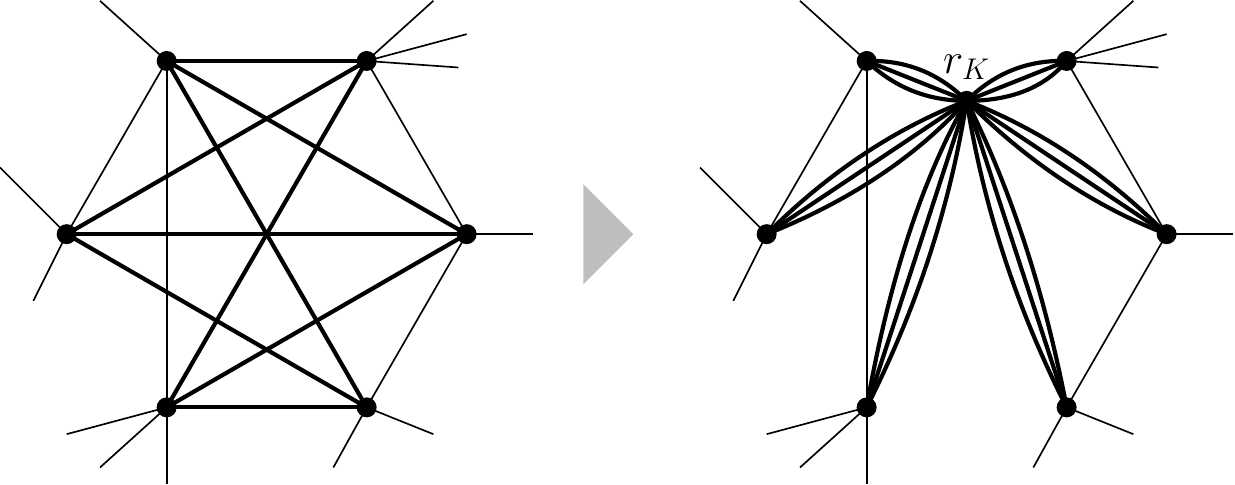}
\caption{
The graph on the left shows 
the edge set $E(K)$ of 
the $t$-regular complete partite graph $K$ by the thick edges, 
while the thin edges belong to $E\setminus E(K)$.
In this example, $K$ is a $3$-regular complete bipartite graph.
The thick edges in the graph on the right depict the newly added three parallel edges between $r_K$ and each vertex $v \in V(K)$.}
\label{fig:reduction}
\end{figure}
\begin{itemize}
    \item Introduce a new vertex $r_K$ for each $K \in \K$, and define the vertex set $V'$ by
    \begin{align*}
        V' = V \cup \{ r_K \colon K \in {\cal K} \}.
    \end{align*}
    \item 
    For each $K \in \K$ and $v \in V(\K)$, 
    introduce new $t$ parallel edges between $r_K$ and $v$, 
    and 
    let $E'_{v,K}$ denote the set of these new $t$ parallel edges.
    Define the edge set $E'$ by
    \begin{align*}
        E'=
        \left(E \cup \bigcup_{K\in \mathcal{K}}\bigcup_{v\in V(K)}E'_{v,K}\right)\setminus \bigcup_{K\in \mathcal{K}}E(K),
    \end{align*}
\end{itemize}

Our input constraint $\F_v \subseteq 2^{\delta_{G'}(v)}$ ($v\in V'$) is constructed as follows:
\begin{itemize}
    \item 
    For each subgraph $K \in \mathcal{K}$, 
    compute a set $D_K \subseteq \mathbb{Z}_+^{V(K)}$
of the degree sequences in the $K$-free $b$-matchings in $K$, i.e.,\ 
\begin{align*}
D_K &= \left\{d_F \in \mathbb{Z}_+^{V(K)} \colon \mbox{$F$ is a $K$-free $b$-matching in $K$} \right\} \\
&= \left\{d_F \in \mathbb{Z}_+^{V(K)} \colon \mbox{$F$ is a $b$-matching in $K$} \right\} \setminus \{ (t, \dots , t)\}.
\end{align*}
Then, for each vertex $v \in V'$, 
define 
${\cal F}_v \subseteq 2^{\delta_{G'}(v)}$ by  
\begin{align}
\label{eq:defnFv}
{\cal F}_v = 
\begin{cases}
\{ F' \subseteq \delta_{G'}(v) \colon |F'| = b(v) \} & \mbox{if $v \in V$}, \\
\{ F' \subseteq \delta_{G'}(v) \colon \left(d_{F'}(u)\right)_{u \in V(K)} \in D_K\} & \mbox{if $v = r_K$ for some $K \in {\cal K}$.}
\end{cases}
\end{align}
\end{itemize}
Note that 
each $D_K$ and each ${\cal F}_v$ can be computed efficiently in a brute force way: 
$|V(K)| = O(t)$ and 
hence $D_K$ has $t^{O(t)}$ elements 
for the fixed integer $t$; 
and 
${\cal F}_v$ has a polynomial size.

Now we have constructed an instance of
the Boolean edge-CSP
consisting of $G'=(V', E')$ and $(\F_v)_{v \in V'}$. 
We first show 
the following claim, 
which implies that 
this instance actually belongs to
\BECG.

\begin{claim}
\label{clm:01}
For each $v \in V'$, 
the set $\{\chi_{F'} \in \mathbb{Z}^{\delta_{G'}(v)} \colon F' \in {\cal F}_v\}$ 
of the characteristic vectors 
of the edge sets in ${\cal F}_v$ 
is 
a constant-parity jump system. 
\end{claim}

\begin{proof}[Proof of Claim \ref{clm:01}]
  If $v \in V$, then the claim follows from the fact that ${\cal F}_v$ is the basis family of a uniform matroid. 
  Suppose that $v = r_K$ for $K \in {\cal K}$. 
  By applying \Cref{thm:parititejump}
  with $G = K$ and ${\cal K} = \{K\}$, we obtain that $D_K$ is a constant-parity jump system. 
  Now, 
  $\{\chi_{F'} \in \mathbb{Z}^{\delta_{G'}(v)} \colon F' \in {\cal F}_v\}$ is 
  obtained from 
  splitting 
  $D_K$ to  
  $\bigcup_{u \in V(K)} E'_{u,K}$ 
  and then 
  taking the intersection with 
  a box $\{x \in \mathbb{R}^{\delta_{G'}(v)} \colon {\bf 0} \le x \le {\bf 1}\}$, 
  and thus 
  is a constant-parity jump system; see Section~\ref{sec:jumpsystem}.  
\end{proof}

It follows from Claim~\ref{clm:01} and Theorem~\ref{thm:EdgeCSP}
that 
the instance 
$(G', (\mathcal{F}_v)_{v\in V'})$ 
belongs to \BECG 
and 
can be solved in polynomial time,
respectively.
Namely, 
we can find an edge set $M' \subseteq E'$ such that 
\begin{equation}
    \delta_{M'}(v) \in {\cal F}_v  \mbox{ for each $v \in V'$} \label{eq:01} 
\end{equation}
or conclude that such $M'$ does not exist 
in polynomial time. 
In what follows, we show that the existence of such an edge set $M' \subseteq E'$ is equivalent to the existence of a ${\cal K}$-free $b$-factor in the original graph $G$. 
\begin{claim}
\label{clm:02}
The graph $G'$ has 
an edge set $M' \subseteq E'$ satisfying~\eqref{eq:01} if and only if 
the original graph $G$ has a ${\cal K}$-free $b$-factor $M \subseteq E$. 
\end{claim}
\begin{proof}[Proof of Claim \ref{clm:02}]
We first show the sufficiency (``if'' part).  
Let $M \subseteq E$ be a ${\cal K}$-free $b$-factor in $G$. 
We construct an edge set $M' \subseteq E'$ satisfying \eqref{eq:01} in the following way. 
For each subgraph $K \in {\cal K}$, let $F_K \subseteq \delta_{G'}(r_K)$ be an edge set in $G'$
composed of exactly $d_{M \cap E(K)}(u)$ parallel edges between $u$ and $r_K$ for each vertex $u \in V(K)$. 
Note that such an edge set $F_K$ must exist, because $M$ is a $b$-factor, $b(u) \le t$, 
and $G'$ has $t$ parallel edges between $u$ and $r_K$.
Now define $M' \subseteq E'$ by 
$$M' = \left(M \setminus \bigcup_{K \in {\cal K}} E(K)\right) \cup \bigcup_{K \in {\cal K}} F_K.$$

Here we show 
that this edge set $M'$ satisfies (\ref{eq:01}).
If $v \in V$, 
it holds that $\delta_{M'}(v) \in {\cal F}_{v}$, 
since 
$|\delta_{M'}(v)| = |\delta_{M}(v)| = b(v)$. 
Let $K \in {\cal K}$ and $v=r_K$.  
The fact that $M$ is $\K$-free 
implies 
$$(d_{M \cap E(K)}(u))_{u\in V(K)} \in D_K.$$
Since $d_{F_K}(u) = d_{M \cap E(K)}(u)$ for each vertex $u \in V(K)$, 
it follows from 
the definition \eqref{eq:defnFv} of $\mathcal{F}_{r_K}$
that $F_K \in {\cal F}_{r_K}$, 
and hence $\delta_{M'}(r_K) = F_K \in {\cal F}_{r_K}$. 
We thus conclude that $M'$ satisfies \eqref{eq:01}. 

We next show the necessity (``only if'' part).  
Let $M' \subseteq E'$ be an edge set satisfying (\ref{eq:01}). 
We construct a $\mathcal{K}$-free $b$-factor $M$ in $G$ in the following manner. 
For each subgraph $K \in {\cal K}$, let $F_K := \delta_{M'}(r_K)$. 
It follows from \eqref{eq:01} that 
$F_K \in {\cal F}_{r_K}$, 
namely, 
there exists a $b$-matching $N_K \subsetneq E(K)$ 
such that 
$d_{N_K}(u)= d_{F_K}(u)$ for each vertex $u\in V(K)$. 
Now define $M\subseteq E$ by 
$$M = \left(M' \setminus \bigcup_{K \in {\cal K}} F_K\right) \cup \bigcup_{K \in {\cal K}} N_K.$$ 

We complete the proof by showing that 
$M$ is a ${\cal K}$-free $b$-factor in $G$. 
Let $v\in V$ be an arbitrary vertex in $G$.
Since 
$d_{F_K}(u) = d_{N_K}(u)$ for each $K \in {\cal K}$ and each $u \in V(K)$, 
it holds that $d_M(v) = d_{M'}(v) =b(v)$, 
where the last equality follows from $\delta_{M'}(v) \in {\cal F}_v$. 
We thus have that $M$ is a $b$-factor. 
Furthermore, since 
$N_K \subsetneq E(K)$ 
for each $K \in {\cal K}$, we conclude that $M$ is ${\cal K}$-free. 
\end{proof}

The proof of Claim \ref{clm:02} provides a polynomial-time construction of 
a ${\cal K}$-free $b$-factor $M$ in $G$ 
from an edge set $M' \subseteq E'$ satisfying (\ref{eq:01}). 
We thus conclude that 
the original instance $(G,b,\mathcal{K})$ of 
\KbFP 
can be solved in polynomial time.  
\end{proof}

By using Theorem \ref{thm:disjointbfactor}, we can give a polynomial-time algorithm
for 
\MKbMP 
under the same assumptions. 

\begin{theorem}
\label{thm:disjointbmatching}
For a fixed positive integer $t$, 
\MKbMP 
can be solved in polynomial time if 
$b(v) \le t$ for each $v \in V$ and 
all the subgraphs in ${\cal K}$ are 
$t$-regular complete partite 
and pairwise edge-disjoint.  
\end{theorem}

\begin{proof}
It follows from \Cref{thm:parititejump} that
the set of the degree sequences of all ${\cal K}$-free $b$-matchings in $G$ 
is a constant-parity jump system. 
Therefore, 
by \Cref{lem:factor2matching} and \Cref{thm:disjointbfactor}, 
we can solve 
\MKbMP 
in polynomial time. 
\end{proof}

We remark that 
Theorem~\ref{thm:disjointtmatching} 
 is immediately derived from 
 Theorem \ref{thm:disjointbmatching}
 by setting $b(v) = t$ for every $v \in V$.

As described in Section \ref{sec:intro}, 
the edge-disjointness of the subgraphs in $\mathcal{K}$ 
is relaxed to the condition~\eqref{cond:star}, which we restate here: 
\begin{enumerate}[label=(\cond)]
\item
The subgraph family $\mathcal{K}$ is partitioned into subfamilies $\mathcal{K}_1, \dots , \mathcal{K}_\ell$ such that 
    \begin{itemize}
        \item for each subfamily $\mathcal{K}_i$ ($i=1,\ldots, \ell$), the number $\left|\bigcup_{K \in \mathcal{K}_i} V(K)\right|$ of vertices is bounded by a fixed constant, and  
        \item for distinct subfamilies $\mathcal{K}_i$ and $\mathcal{K}_j$ ($i,j \in \{1, \dots , \ell\})$ and for 
        each pair of subgraphs $K \in \mathcal{K}_i$ and $K' \in \mathcal{K}_j$, 
        it holds that $K$ and $K'$ are edge-disjoint. 
    \end{itemize}
\end{enumerate}

\begin{theorem}
\label{thm:starbmatching}
    For a fixed positive integer $t$, 
    \KbFP and \MKbMP 
    can be solved in polynomial time
    if $b(v) \le t$ for each $v \in V$ and 
    $\K$ is a family of $t$-regular complete partite subgraphs of $G$ 
    and satisfies the condition~\eqref{cond:star}. 
\end{theorem}

\begin{proof}
It follows from \Cref{thm:parititejump} and \Cref{lem:factor2matching} 
that 
\MKbMP 
can also be solved in polynomial time 
if 
\KbFP 
is so. 
Hence, 
below we prove that 
\KbFP 
can be solved in polynomial time in a similar way to \Cref{thm:disjointbfactor}.

Let $(G,b,\K)$ be
an instance of 
\KbFP,
where 
$G=(V,E)$, $b\in \Z_+^V$, and $\cal K$ 
is a family of subgraphs in $G$ satisfying the condition~\eqref{cond:star}. 
Let $\mathcal{K}_1, \dots , \mathcal{K}_\ell$ be the partition of $\mathcal{K}$ in the condition~\eqref{cond:star}. 

For each $i \in \{1, \dots , \ell\}$, execute the following procedure. 
Let $H_i$ be the graph defined as the union of all $K \in \K_i$,
i.e.,
\begin{align*}
    H_i := \left(\bigcup_{K \in \K_i} V(K), \bigcup_{K \in \K_i} E(K)\right).
\end{align*}
Then, 
\begin{itemize}
\item 
add a new vertex $r_i$ and $t$ parallel edges between $r_i$ and $v$ for each $v \in V(H_i)$, and remove the original edges in $E(H_i)$; and
\item 
compute a set 
$D_{H_i} \subseteq \mathbb{Z}_+^{V(H_i)}$
of the degree sequences in the $\K_i$-free $b$-matchings in $H_i$, i.e.,\ 
\begin{align*}
D_{H_i} = \left\{d_F \in \mathbb{Z}_+^{V(H_i)} \colon \mbox{$F$ is a $\K_i$-free $b$-matching in $H_i$} \right\}.
\end{align*} 
\end{itemize}

For each $i \in \{1,\ldots, \ell\}$, 
it follows from Theorem~\ref{thm:parititejump} that the set $D_{H_i}$ is a constant-parity jump system. 
We also remark that 
$D_{H_i}$ can be computed efficiently in a brute force way, 
since 
$|V(H_i)|$ and $t$ are bounded by a fixed constant. 

Now, by the same argument as in the proof of \Cref{thm:disjointbfactor}, 
we can solve 
\KbFP 
in polynomial-time
with the aid of \Cref{thm:EdgeCSP}. 
\end{proof}

We conclude this section 
by showing that a subgraph family $\K$ with a certain laminar structure described below satisfies the condition~\eqref{cond:star}.

\begin{corollary}\label{cor:example}
For a fixed positive integer $t$, 
\KbFP and \MKbMP 
can be solved in polynomial time
if $b(v) \le t$ for each $v \in V$ and $\K$ is a family of $t$-regular complete partite subgraphs of $G$
satisfying $E(K) \cap E(K') = \emptyset$,
$V(K) \subseteq V(K')$,
or $V(K) \supseteq V(K')$
for each pair of subgraphs $K, K' \in \K$.
\end{corollary}
\begin{proof}
We construct a partition of $\K$ 
certifying that $\K$ satisfies the condition \eqref{cond:star}.  
Then the corollary immediately follows from  \Cref{thm:starbmatching}. 

Let $\mathcal{X}^*$ be the family of all inclusionwise maximal sets in $\{ V(K) \colon K \in \K \}$.
For each vertex set $X \in \mathcal{X}^*$, 
define a subfamily $\K_X$ of $\K$ by $\K_X := \{ K \in \K \colon V(K) \subseteq X \}$.
It suffices to show 
that $E(K) \cap E(K') = \emptyset$ 
for each pair of distinct vertex sets $X, X' \in \mathcal{X}^*$ and for 
each pair of subgraphs 
$K \in \K_X$ and $K' \in \K_{X'}$; 
this implies that $\K_X$ ($X \in \mathcal{X}^*$) form a partition of $\K$ satisfying the condition~\eqref{cond:star}.

Suppose to the contrary that
$E(K) \cap E(K') \neq \emptyset$
for some distinct vertex sets $X, X' \in \mathcal{X}^*$ 
and for some subgraphs $K \in \K_X$ and $K' \in \K_{X'}$. 
It follows from  the assumption of $\K$ 
that $V(K) \subseteq V(K')$ or $V(K) \supseteq V(K')$. 
Without loss of generality, 
assume $V(K) \subseteq V(K')$.
Let $K_X \in \K_X$ (resp.~$K_{X'} \in \K_{X'}$) be a $t$-regular complete partite graph attaining $V(K_X) = X$ (resp.~$V(K_{X'}) = X'$).
It follows from the maximality of $X$ and $X'$ that $V(K_X) \not\subseteq V(K_{X'})$ and $V(K_X) \not\supseteq V(K_{X'})$.

In the following, we prove that $E(K_X) \cap E(K_{X'}) \neq \emptyset$,
which contradicts 
the assumption of $\K$.
Define a vertex set $Y$ by $Y := V(K_X) \cap V(K_{X'})$.
It is derived from  $V(K) \subseteq X = V(K_X)$ and $V(K) \subseteq V(K') \subseteq X' = V(K_{X'})$ 
that $V(K) \subseteq Y$.
It then follows that  $|Y| \geq t+1$,
which implies that both of the induced subgraphs $K_X[Y]$ and $K_{X'}[Y]$ are complete partite graphs having at least two color classes.
Since the complement of $K_X[Y]$ is the disjoint union of (at least two) complete graphs,
it is disconnected.
On the other hand, $K_{X'}[Y]$ is connected.
Hence we have $E(K_X[Y]) \cap E(K_{X'}[Y]) \neq \emptyset$,
implying that  $E(K_X) \cap E(K_{X'}) \neq \emptyset$.
\end{proof}

\section{Degree Bounded Graphs}
\label{sec:boundeddegree}
In this section, we consider the case where the maximum degree of $G$ is at most $2t-1$.

\begin{theorem}
\label{thm:degreebfactor}
For a fixed positive integer $t$,  
\KbFP 
can be solved in polynomial time if 
the maximum degree of $G$ is at most $2t-1$, 
$b(v) \le t$ for each $v \in V$, and 
all the subgraphs in 
${\cal K}$ are $t$-regular complete partite. 
\end{theorem}

\begin{proof}
If $t=1$, then the problem is trivial, 
because the maximum degree is one and 
a $t$-regular complete partite subgraph must be composed of a single edge. 
Therefore, it suffices to consider the case where $t \ge 2$. 

Without loss of generality, we may assume that each subgraph $K \in {\cal K}$ satisfies 
\begin{equation}
b(v) = t \mbox{ for each vertex } v \in V(K), \label{eq:assumption1}
\end{equation}
since otherwise we can remove $K$ from ${\cal K}$.

Define a vertex subset family $\mathcal{X}\subseteq 2^V$ by ${\cal X} = \{V(K) \colon K \in {\cal K}\}$. 
Construct a subfamily ${\cal X}^* \subseteq {\cal X}$ 
of disjoint vertex subsets in $\mathcal{X} $ in the following manner: 
start with ${\cal X}^* = \emptyset$; and  while there exists 
a set in ${\cal X}$ disjoint from every set in ${\cal X}^*$, 
add an inclusionwise maximal one to ${\cal X}^*$. 
We denote ${\cal X}^* = \{X_1, X_2, \dots , X_\ell \}$. 
It follows from 
the construction that ${\cal X}^* \subseteq {\cal X}$ satisfies the following property: 
\begin{equation}
\mbox{for each $X \in {\cal X} \setminus {\cal X}^*$, there exists 
$X_i \in \mathcal{X}^*$
such that  $X \cap X_i \neq \emptyset$ and $X_i \not\subseteq X$.} \label{eq:assumption3}
\end{equation}

For each $X_i \in \mathcal{X}^*$, 
let ${\cal K}_i = \{K \in {\cal K} \colon V(K) \subseteq X_i\}$ 
and 
let $H_i$ be 
the union of all subgraphs in ${\cal K}_i$, i.e., $$H_i = \left(X_i,  \bigcup_{K \in {\cal K}_i} E(K)\right).$$ 
Let ${\cal K}^* = \bigcup_{i=1}^\ell {\cal K}_{i}$. 
Note that 
${\cal K}_{1}, \dots , {\cal K}_{\ell}$  form a partition of $\K^*$, 
and they 
satisfy 
the condition \eqref{cond:star}. 

By using Theorem~\ref{thm:starbmatching}, in polynomial time, 
we can find a ${\cal K}^*$-free $b$-factor $M$ in $G$ or 
conclude that 
$G$ has no ${\cal K}^*$-free $b$-factor. 
In the latter case, 
we can conclude that 
$G$ has no ${\cal K}$-free $b$-factor, because 
$\K^*$ is a subfamily of $\K$. 
In the former case, 
we transform $M$ into a ${\cal K}$-free $b$-factor as shown in the following claim.

\begin{claim}\label{clm:reduction}
Given a ${\cal K}^*$-free $b$-factor $M$ in $G$, 
we can construct a ${\cal K}$-free $b$-factor in polynomial time. 
\end{claim}

\begin{proof}[Proof of \Cref{clm:reduction}]
For a $b$-factor $M$ in $G$, 
define a subgraph family $\K(M)$ by $$\K(M) = \{ K \in {\cal K} \colon  E(K) \subseteq M \},$$
the set of forbidden subgraphs included in $M$. 
Obviously, $M$ is ${\cal K}$-free if and only if $\K(M) = \emptyset$. 
In what follows, 
given a $\K^*$-free $b$-factor $M$, we modify $M$ so that $\K(M)$ becomes smaller. 

Let $M$ be a ${\cal K}^*$-free $b$-factor and suppose that $\K(M) \neq \emptyset$.  
Then, there exists a subgraph $K \in {\cal K} \setminus {\cal K}^*$ such that 
$K\in \K(M)$, i.e.,\ 
$E(K) \subseteq M$. 
It follows from $K \in {\cal K} \setminus {\cal K}^*$ that $V(K) \in {\cal X} \setminus {\cal X}^*$. 
Then,  (\ref{eq:assumption3}) implies that  
there exists $X_i\in \mathcal{X}^*$ such that 
\[V(K) \cap X_i \neq \emptyset \quad \mbox{and} \quad X_i \not\subseteq V(K).\] 
It holds that $ X_i= V(K^*)$ for some 
$K^* \in {\cal K}_i$, 
which follows from  
the construction of ${\cal X}^*$ and 
the definition of ${\cal K}_i$. 
We thus obtain 
$$V(K) \cap V(K^*) \neq \emptyset \quad \mbox{and} \quad V(K^*) \not\subseteq V(K).$$ 

Take a vertex $u$ in $V(K) \cap V(K^*)$. 
Since $|\delta_{K}(u)| = |\delta_{K^*}(u)| = t$ and $|\delta_G(u)| \le 2t-1$, 
there exists an edge $e \in \delta_{K}(u) \cap \delta_{K^*}(u)$, in particular $e \in E(K) \cap E(K^*)$. 
We denote $e = u u'$. Note that $e \in M$ since $E(K) \subseteq M$.

Since $V(K^*) \not\subseteq V(K)$, there exists a vertex $v \in V(K^*) \setminus V(K)$. 
From (\ref{eq:assumption1}) and $K^* \in \mathcal{K}$, we obtain $|\delta_M(v)| = b(v) = t$. 
It then follows from $|\delta_G(v)| \le 2t-1$ and $|\delta_{K^*}(v)| = t$ that $\delta_M(v) \cap \delta_{K^*}(v) \neq \emptyset$, that is, 
there exists an edge $e^* \in \delta_{K^*}(v)$ contained in $M$.  
We denote $e^* = v v'$. 
Since $K$ is a connected component of the subgraph induced by $M$ (see \Cref{rmk:connected}), 
it holds that $v' \in V(K^*) \setminus V(K)$; see Figure~\ref{fig:degreeboundedcase}.
\begin{figure}
\centering
\includegraphics[width=6cm]{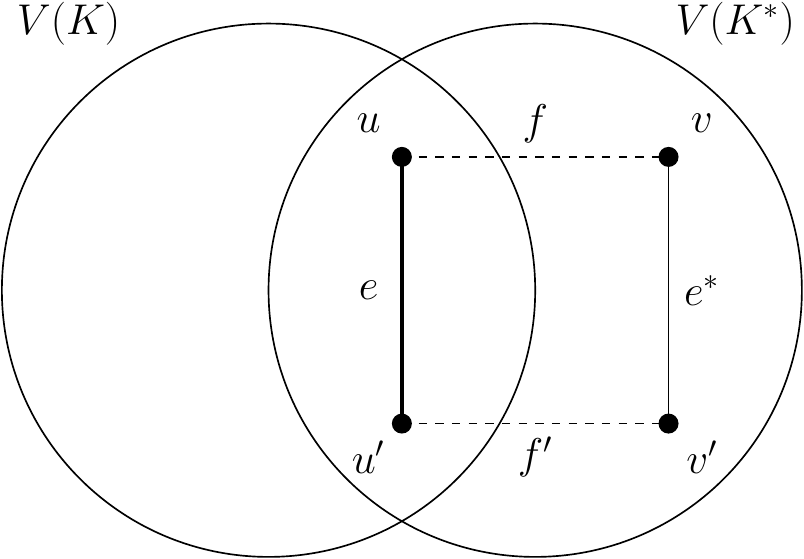}
\caption{All of the edges are in $E(K^*)$
and, particularly,
all of the solid edges are in $M$.
The solid bold edge is in $E(K^*) \cap E(K)$
and
the other thin edges are in $E(K^*) \setminus E(K)$.}
\label{fig:degreeboundedcase}
\end{figure}

Since $e, e^* \in E(K^*)$ and $K^*$ is a complete partite graph, 
$u$ and $u'$ are contained in different color classes of $K^*$, and 
so are $v$ and $v'$. 
This shows that $K^*$ contains two edges:  
$uv$ and $u'v'$; or $uv'$ and $u'v$. 
By symmetry, assume that $f=u v$ and $f'=u' v'$ are contained in $K^*$; see Figure~\ref{fig:degreeboundedcase} again. 
Note that $f$ and $f'$ are not contained in $M$, because $\delta_M(u) = \delta_K(u)$ and $\delta_M(u') = \delta_K(u')$ hold. 

Define $M' = (M \setminus \{e, e^*\}) \cup \{f , f'\}$, which is also a $b$-factor. 
In what follows, we prove that $M'$ is the desired $\K^*$-free $b$-factor, 
i.e., 
$\K(M') \subsetneq \K(M)$. 
Since $K \not\in \K(M')$, 
it suffices to show that $\K(M') \subseteq \K(M)$.

Assume to the contrary that there exists a subgraph $K' \in \K(M') \setminus \K(M)$. 
Then, $K'$ must contain at least one of $f$ and $f'$, 
and without loss of generality assume that $f \in E(K')$.  
Since $K - e$ is connected by $t \ge 2$ and $M'$ contains $(E(K) \setminus \{e\}) \cup \{f\}$ by $e^* \notin E(K)$,
it follows from \Cref{rmk:connected} that $V(K) \cup \{v\}$ is contained in $K'$, 
in particular $u, u', v \in V(K')$. 

Since all the edges in $\delta_M(u')$ are contained in $K$ and $v \not\in V(K)$, 
$M$ has no edge connecting $u'$ and $v$, and neither does $M'$. 
It then follows from $K'\in \K(M')$, i.e.,\ $E(K') \subseteq M'$,  
that $u' v \not\in E(K')$.  
Since $e$ is the only edge in $M$ connecting $u$ and $u'$, we have $u u' \not\in M'$, 
which implies that $u u' \not\in E(K')$. 
It now follows from $u' v, u u' \not\in E(K')$ 
that 
$u, u'$ and $v$ are contained in the same color class of $K'$, 
since $K'$ is complete partite. 
This contradicts the fact that $K'$ contains $f=uv$, 
and thus we conclude that $\K(M') \subsetneq \K(M)$. 

By repeating the above procedure, 
we obtain a $b$-factor $M$ with $\K(M) = \emptyset$, 
i.e.,\ 
$M$ is $\K$-free. 
It is straightforward to see that this procedure can be executed in polynomial time, which completes the proof. 
\end{proof}

Therefore, we conclude that 
\KbFP 
can be solved in polynomial time. 
\end{proof}

From \Cref{thm:degreebfactor}, 
we can derive the following theorem by applying the same argument as \Cref{thm:disjointbmatching}.

\begin{theorem}
\label{thm:degreebmatching}
For a fixed positive integer $t$, 
\MKbMP 
can be solved in polynomial time if
the maximum degree of $G$ is at most $2t-1$, 
$b(v) \le t$ for each $v \in V$, and 
all the subgraphs in ${\cal K}$ are $t$-regular complete partite. 
\end{theorem}

From \Cref{thm:degreebmatching}, 
we immediately obtain \Cref{thm:degreetmatching}  by setting $b(v) = t$ for every $v \in V$.

\section{Generalization: Forbidden Degree Sequences}
\label{sec:generalization}

In this section, we extend 
Theorems~\ref{thm:disjointbfactor}, \ref{thm:disjointbmatching}, and \ref{thm:starbmatching} 
so that 
the forbidden structure 
is not an edge-disjoint family of $t$-regular subgraphs but 
that of the subgraphs with specified degree sequences. 
A similar problem of \emph{general factors} is of classical and recent interest~\cite{Cor88,DP18,Kob23,Lov73,Seb93,SZ23}.

Let $G=(V, E)$ be a graph 
and 
$\mathcal{H}$ be a family of subgraphs of $G$. 
Each subgraph $H\in \mathcal{H}$ 
is associated with 
a set of degree sequences 
$\overline{D}_H \subseteq \mathbb{Z}_{+}^{V(H)}$.  
Define an edge subset family $\mathcal{F}_H \subseteq 2^{E(H)}$ 
by 
\[
\mathcal{F}_H =\{ F \subseteq E(H) \colon d_F \in \overline{D}_H \}, 
\]
and 
subgraph families  
\begin{align*}
{\cal K}_H = \{(V(H), F)\colon F \in \mathcal{F}_H\}, \quad {\cal K}_{\cal H} = \bigcup_{H \in {\cal H}} {\cal K}_H. 
\end{align*}

We are interested in $\mathcal{K}_{\mathcal{H}}$-free $b$-matchings. 
Namely, 
for a subgraph $H\in \mathcal{H}$, 
$\overline{D}_H$ represents the set of forbidden degree sequences on $V(H)$, 
and 
$\mathcal{K}_H$ represents 
the family of the forbidden subgraphs of $H$, 
i.e., 
those  
attaining the degree sequences in  $\overline{D}_H$. 

We now 
extend Theorems~\ref{thm:disjointbfactor}, \ref{thm:disjointbmatching}, and \ref{thm:starbmatching} in the following way.

\begin{theorem}
\label{thm:generalization}
\KbFP and \MKbMP 
can be solved in polynomial time 
if the following conditions are satisfied: 
\begin{enumerate}
\item 
\label{ENU01}
$b(v)$ is bounded by a fixed constant for each $v \in V$; and 
\item
\label{ENU02}
${\cal K} = {\cal K}_{\cal H}$ 
for an edge-disjoint family ${\cal H}$ of subgraphs of $G$ such that, 
for each $H \in {\cal H}$,  
\begin{enumerate}
\item
$|V(H)|$ is bounded by a fixed constant, and 
\item
$D_H := \{d_F \colon \mbox{$F$ is a $b$-matching in $H$} \} \setminus \overline{D}_H$ is a constant-parity jump system. 
\end{enumerate}
\end{enumerate}
\end{theorem}

\begin{proof}
    Suppose that 
    the conditions \ref{ENU01} and \ref{ENU02} are satisfied. 
    By following the proof of \Cref{thm:disjointbfactor}, we see that 
    \KbFP 
    can be solved in polynomial time. 

    We now prove that 
    \MKbMP
    can be solved in polynomial time. 
    Define $J \subseteq \mathbb{Z}_+^V$ as the set of the degree sequences of all ${\cal K}$-free $b$-matchings in $G$. 
    In order to apply \Cref{lem:factor2matching}, 
    we show that $J$ is a constant-parity jump system.  
    
    Define $J_0\subseteq \Z_+^V$ by 
    $$J_0 = \left\{ d_F \in \mathbb{Z}_+^V \colon F \subseteq E \setminus \bigcup_{H \in {\cal H}} E(H) \right\},$$ which is a constant-parity jump system. 
    For each subgraph $H \in {\cal H}$, 
    regard $D_H \subseteq \mathbb{Z}_+^{V(H)}$ as a subset of $\mathbb{Z}_+^V$ by setting $x(v) = 0$ 
    for each $x \in D_H$ and $v \in V \setminus V(H)$.  
    Then, $J$ is obtained from 
    $J_0$ by taking 
    the Minkowski sum with  $ D_H$ for all $H \in {\cal H}$, 
    and 
    then taking the intersection with a box $\{x \in \mathbb{R}^V \colon {\bf 0} \le x \le b \}$.   
    This shows that $J$ is a constant-parity jump system.  

    Thus, by applying \Cref{lem:factor2matching}, 
    we conclude that 
    \MKbMP 
    can be solved in polynomial time. 
\end{proof}

Observe that 
Theorems~\ref{thm:disjointbfactor} and~\ref{thm:disjointbmatching} are exactly 
special cases 
of Theorem \ref{thm:generalization}, where 
$H$ is a $t$-regular complete partite graph and $\overline{D}_H = \{ (t, t, \dots , t) \}$ 
for each $H \in \mathcal{H}$. 
Observe also that 
Theorem~\ref{thm:starbmatching} is a special case 
of Theorem \ref{thm:generalization}, where 
$\mathcal{H}=\{H_1, \dots , H_\ell\}$.

We conclude this paper with a few applications of \Cref{thm:generalization}.

\begin{example}
\label{EX01}
Suppose that each subgraph $H \in {\cal H}$ is obtained from $K_5$ by removing a matching of size two (which is unique up to isomorphism). 
Let $\overline{D}_H = \{(2, \dots , 2) \}$ for each $H \in {\cal H}$ and let $b(v) = 2$ for each $v \in V$. 
It then follows that $D_H = \{d_F \colon \mbox{$F$ is a $2$-matching in $H$} \} \setminus \{ (2, \dots , 2)\}$ is a constant-parity jump system. 
It also follows that the subgraph family ${\cal K}_H$ consists of cycles of length five in $H$. 
Now 
Theorem~\ref{thm:generalization} shows that we can find a maximum $2$-matching which does not contain the cycles of length five 
in $H$ for each subgraph $H \in \mathcal{H}$. 
This is an interesting contrast to the fact that finding a maximum $C_5$-free $2$-matching in graphs is NP-hard (see \cite{CP80}). 
To the best of our knowledge, 
this is the first polynomially solvable class of the restricted 
$2$-matching problem 
excluding cycles of length five. 
\end{example}

\begin{example}
\label{EX02}
Suppose that each subgraph $H \in {\cal H}$ is obtained from $K_{3,3}$ by removing an edge (which is unique up to isomorphism). 
Let $\overline{D}_H = \{(2, \dots , 2) \}$ for each $H \in {\cal H}$ and let $b(v) = 2$ for each $v \in V$. 
It then follows that $D_H = \{d_F \colon \mbox{$F$ is a $2$-matching in $H$} \} \setminus \{ (2, \dots , 2)\}$ is a constant-parity jump system. 
It also follows that the subgraph family ${\cal K}_H$ consists of the cycles of length six in $H$. 
Now Theorem~\ref{thm:generalization} shows that we can find a maximum $2$-matching which does not contain the cycles of length six in $H$ for each 
subgraph $H\in \mathcal{H}$. 
This is an interesting contrast to the fact that finding a maximum $C_6$-free $2$-matching in bipartite graphs is NP-hard (Geelen, see \cite{Fra03,Kir09}).
Note also that such a graph $H$ (i.e., $K_{3,3} - e$) is discussed by Takazawa~\cite{Tak17DO} as an example of so called \emph{Hamilton-laceable graphs} \cite{Sim78}. 
\end{example}

\section*{Acknowledgments}
The first author was supported by JSPS KAKENHI Grant Numbers JP20K23323, JP20H05795, JP22K17854.
The second author was supported by JSPS KAKENHI Grant Numbers JP20K11692, JP20H05795, JP22H05001. 
The third author was supported by
JSPS KAKENHI Grant 
Number JP20K11699.

\bibliographystyle{abbrv}

\end{document}